\documentclass[a4paper,12pt]{amsart}
\usepackage{eurosym}
\usepackage{geometry}
\usepackage{amsmath,amssymb,amsfonts,latexsym}
\usepackage{graphicx}
\usepackage{algorithm}
\usepackage{algorithmicx}
\usepackage{algpseudocode}
\usepackage[table]{xcolor}
\usepackage{float}
\usepackage{graphics}
\usepackage{subcaption}
\usepackage{listings,xcolor}

\newtheorem{theorem}{Theorem}[section]

\newtheorem{corollary}[theorem]{Corollary}

\newtheorem{remark}{Remark}

\begin{document}
\title[ {\ New integral formulas and identities involving special numbers}...%
]{{\Large New integral formulas and identities involving special numbers and
functions derived from certain class of special combinatorial sums}}
\author{YILMAZ SIMSEK}
\address{Department of Mathematics, Faculty of Science University of Akdeniz
TR-07058 Antalya-TURKEY}
\email{ysimsek@akdeniz.edu.tr}

\begin{abstract}
\vspace{2mm}By applying $p$-adic integral on the set of $p$-adic integers in 
\cite{SimsekMTJPAM2020} (Interpolation Functions for New Classes Special
Numbers and Polynomials via Applications of p-adic Integrals and Derivative
Operator, Montes Taurus J. Pure Appl. Math. 3 (1), ...--..., 2021 Article
ID: MTJPAM-D-20-00000), we constructed generating function for the special
numbers and polynomials involving the following combinatorial sum and
numbers:%
\begin{equation*}
y(n,\lambda )=\sum_{j=0}^{n}\frac{(-1)^{n}}{(j+1)\lambda ^{j+1}\left(
\lambda -1\right) ^{n+1-j}}
\end{equation*}

The aim of this paper is to use the numbers $y(n,\lambda )$ to derive some
new and novel identities and formulas associated with the Bernstein basis
functions, the Fibonacci numbers, the Harmonic numbers, the alternating
Harmonic numbers, binomial coefficients and new integral formulas for the
Riemann integral. We also investigate and study on open problems involving
the numbers $y(n,\lambda )$ in \cite{SimsekMTJPAM2020}. Moreover, we give
relation among the numbers $y\left(n,\frac{1}{2}\right)$, the Digamma
function, and the Euler constant. Finally, we give conclusions for the
results of this paper with some comments and observations.


\vspace{2mm}

\end{abstract}

\maketitle


\section{Introduction}

Due to the works in \cite{3Euler} and \cite{2Euler}, we see that Daniel
Bernoulli and Goldbach were studied not only on interpolating a sequence,
but also on the partial sums of the harmonic series. But these studies were
not without much success. That is, Bernoulli and Goldbach were working the
following partial sums of the harmonic series%
\begin{equation*}
f(v)=\sum_{j=1}^{v}\frac{1}{j}
\end{equation*}%
where $v$ is a positive integer. In one of his studies, Euler\ found many
formulas on these topics involving the Digamma function, and the Euler
constant, interpolating a sequence, partial sums of the harmonic series. We
can find these clues in Euler's letter to Christian Goldbach dated October
13, 1729. In this letter, Euler mentioned the formulas he found and sent the
following formula, one of them, to Goldbach without any proof: 
\begin{equation*}
f\left(\frac{1}{2}\right) =2-2\log (2)
\end{equation*}
(\textit{cf}. \cite[p.136]{2Euler}, \cite[pp. 283--284]{3Euler}). In \cite[%
p.136]{2Euler}, \cite[pp. 283--284]{3Euler}, the following formula was
given: 
\begin{equation*}
f\left( \frac{1}{2}\right) =\int\limits_{0}^{1}\frac{1-\sqrt{x}}{1-x}dx.
\end{equation*}%
The motivation of this paper is of come from the spirit of historical
mathematical works of Bernoulli, Euler and Goldbach with the help of
applications of the following combinatorial numbers which was defined in 
\cite{SimsekMTJPAM2020}: 
\begin{equation}
y(n,\lambda )=\sum_{j=0}^{n}\frac{(-1)^{n}}{(j+1)\lambda ^{j+1}\left(
\lambda -1\right) ^{n+1-j}}  \label{ynldef}
\end{equation}

It would be very appropriate to give a brief information about how the
numbers $y(n,\lambda )$ can be found. By applying $p$-adic integral on the
set of $p$-adic integers in \cite{SimsekMTJPAM2020}, we constructed
generating function for the special numbers and polynomials involving
combinatorial sums and many special numbers and polynomials. When we
investigated and studied on interpolation functions for these new classes
special numbers and polynomials with series algebraic operations, we
performed on one of the these series multiplications with its inspiration,
the numbers $y(n,\lambda )$ given in equation (\ref{ynldef}) were defined.
In \cite{SimsekMTJPAM2020}, some of their properties for these numbers were
discussed and given.

In \cite{SimsekMTJPAM2020}, we have put forward the following open problems
for the numbers $y(n,\lambda )$\textbf{:}

\begin{enumerate}
\item \textit{One of the first questions that comes to mind what is
generating function for the numbers }$y(n,-1)$ and the numbers $y(n,\lambda
) $\textit{.}

\item \textit{Some of the other questions are what are the special families
of numbers the numbers }$y(n,-1)$\textit{\ are related to.}

\item \textit{What are the combinatorial applications of the numbers }$%
y(n,-1)$\textit{.}

\item \textit{Can we find a special arithmetic function representing this
family of numbers?}
\end{enumerate}

Partial solutions to some of the above open problems can be explored and
studied. By using (\ref{ynldef}), we can prove the following theorems
including formulas and identities related to the numbers $y(n,\lambda )$,
the harmonic numbers, the Fibonacci numbers, the Bernstein basis functions,
and the Riemann integral:

\begin{theorem}
\label{Th.berD} Let $n\in \mathbb{N}_{0}$. Then we have%
\begin{equation}
y(n,\lambda )=\frac{-1}{\lambda }\sum_{j=0}^{n}\sum_{k=0}^{j}\binom{n+1}{j}%
\frac{B_{k}S_{1}(j,k)}{j!B_{j}^{n+1}\left( \lambda \right) }.  \label{91c1}
\end{equation}%
where $B_{j}^{n}\left( \lambda \right) $, $B_{j}$, and $S_{1}(j,k)$ denote
the Bernstein basis function and the Bernoulli numbers, and the Stirling
numbers of the first kind, respectively.
\end{theorem}

\begin{theorem}
\label{Th. 1a}Let $n\in \mathbb{N}_{0}$. Then we have%
\begin{equation*}
y(n,\lambda )=\frac{(-1)^{n}}{\left( \lambda -1\right) ^{n+2}}%
\int\limits_{0}^{\frac{\lambda -1}{\lambda }}\frac{1-x^{n+1}}{1-x}dx.
\end{equation*}
\end{theorem}

\begin{theorem}
\label{Th Hn}Let $n\in \mathbb{N}_{0}$. Then we have%
\begin{equation}
y\left( n,\frac{1}{2}\right) =2^{n+2}\left( H_{\left[ \frac{n}{2}\right]
}-H_{n}+\frac{(-1)^{n+1}}{n+1}\right) .  \label{91c}
\end{equation}%
where $H_{n}$ denotes the harmonic numbers and $H_{0}=0$.
\end{theorem}

\begin{theorem}
\label{Th.Int}Let $n\in \mathbb{N}_{0}$. Then we have%
\begin{equation}
y\left( n,\frac{1}{2}\right) =\frac{1}{2^{n+2}}\int\limits_{0}^{1}\frac{%
1+(-1)^{n}x^{n+1}}{1+x}dx.  \label{91c2}
\end{equation}
\end{theorem}

\begin{theorem}
\label{Th.Int2}Let $n\in \mathbb{N}_{0}$. Then we have%
\begin{equation}
\int\limits_{0}^{1}\frac{1-x^{\left[ \frac{n}{2}\right] }}{1-x}%
dx-\int\limits_{0}^{1}\frac{1-x^{n}}{1-x}dx=\frac{y\left( n,\frac{1}{2}%
\right) }{2^{n+2}}+\frac{(-1)^{n}}{n+1},  \label{91c3}
\end{equation}

where $\left[ x\right] $ denotes the integer part of $x$.
\end{theorem}

\begin{theorem}
\label{Th. Fib}Let $n\in \mathbb{N}_{0}$. Then we have%
\begin{equation}
y\left( n,\frac{1+\sqrt{5}}{2}\right) =\sum_{j=0}^{n}\frac{\left( -1\right)
^{j+1}}{(j+1)\left( \frac{1+\sqrt{5}}{2}F_{j+1}+F_{j}\right) \left( \frac{1-%
\sqrt{5}}{2}F_{n-j+1}+F_{n-j}\right) },  \label{91c4}
\end{equation}%
where $F_{n}$ denotes the Fibonacci numbers.
\end{theorem}

Combinatorial sums and numbers, special numbers and polynomials are the main
topics that can be used in almost all areas of mathematics. In addition,
they are also one of the most frequently used topics in applied sciences
involving mathematical modeling and algorithmic solutions for real world
problems.

We give the following some basic standard notations, formulas, and
definitions:

Let $\mathbb{N}$, $\mathbb{%
\mathbb{Z}
}$, $\mathbb{R}$, and $\mathbb{C}$ denote the set of natural numbers, the
set of integer numbers, the set of real numbers and the set of complex
numbers, respectively, and also $\mathbb{N}_{0}=\mathbb{N}\cup \left\{
0\right\} $. Supposing that for $z\in \mathbb{C}$ with $z=x+iy$ ($x,y\in 
\mathbb{R}$); $Re(z)=x$ and $Im(z)=y$ and also $\log z$ denotes the
principal branch of the many-valued function $Im(\log z)$ with the imaginary
part of $\log z$ constrained by%
\begin{equation*}
-\pi <Im(\log z)\leq \pi ,
\end{equation*}%
and $\log e=1$.%
\begin{equation*}
0^{n}=\left\{ 
\begin{array}{cc}
1, & n=0 \\ 
0, & n\in \mathbb{N}%
\end{array}%
\right.
\end{equation*}%
The Bernoulli polynomials, $B_{n}\left( x\right) $, are defined by the
following generating function:%
\begin{equation}
\frac{t}{e^{t}-1}e^{xt}=\sum_{n=0}^{\infty }B_{n}\left( x\right) \frac{t^{n}%
}{n!},  \label{ApostolBern}
\end{equation}%
where $\left\vert t\right\vert <2\pi $ (\textit{cf}. \cite{3Euler}-\cite%
{SrivatavaChoi}; and references therein).

When $x=0$ in (\ref{ApostolBern}), we have%
\begin{equation*}
B_{n}\left( 0\right) =B_{n}
\end{equation*}%
where $B_{n}$ denotes the Bernoulli numbers (\textit{cf}. \cite{3Euler}-\cite%
{SrivatavaChoi}; and references therein).

The Bernstein basis function $B_{j}^{n}\left( \lambda \right) $ is defined by%
\begin{equation*}
B_{j}^{n}\left( \lambda \right) =\binom{n}{j}\lambda ^{j}(1-\lambda )^{n-j}
\end{equation*}%
where $j\in \left\{ 0,1,2,\ldots ,n\right\} $ (\textit{cf}. \cite{Lorentz}).

The Stirling numbers of the second kind, $S_{1}\left( n,k\right) $, are
defined by the following generating function:%
\begin{equation}
\frac{\left( \log (1+z)\right) ^{k}}{k!}=\sum_{n=0}^{\infty }S_{1}\left(
n,k\right) \frac{z^{n}}{n!}  \label{Sitirling1}
\end{equation}%
with $S_{1}\left( n,k\right) =0$ if $k>n$, and $k\in \mathbb{N}_{0}$ (%
\textit{cf}. \cite{3Euler}-\cite{SrivatavaChoi}; and references therein).

In \cite{KimDahee} and \cite{KimV}, with the aid of the Volkenborn integral
on the set of $p$-adic integers, Kim defined the Daehee numbers $D_{n}$ by
the following generating function:%
\begin{equation*}
\frac{\log (1+z)}{z}=\sum_{n=0}^{\infty }D_{n}\frac{z^{n}}{n!},
\end{equation*}%
where $z\neq 0$ and $\left\vert z\right\vert <1$ (\textit{cf}. \cite%
{KimDahee}, \cite{KimV}).

By using the above equation, the following explicit formula for the Daehee
numbers $D_{n}$ is found: 
\begin{equation}
D_{n}=(-1)^{n}\frac{n!}{n+1}  \label{D}
\end{equation}%
(\textit{cf}. \cite{KimDahee}, \cite{KimV}).

Substituting $t=\log (1+z)$, $x=0$, and $z\neq 0$\ into (\ref{ApostolBern}),
by using Riordan method given in \cite[p. 45, Exercise 19 (b)]{5Riardon},
after some elementary calculations, we have%
\begin{equation}
\log (1+z)=z\sum_{n=0}^{\infty }B_{n}\frac{(\log (1+z))^{n}}{n!}.  \label{d8}
\end{equation}%
Combining the above equation with (\ref{Sitirling1}), after some elementary
calculations, $|z|<1$ we get%
\begin{equation*}
\sum\limits_{m=1}^{\infty }\left( -1\right) ^{m+1}\frac{z^{m-1}}{m}%
=\sum_{m=0}^{\infty }\sum_{n=0}^{m}B_{n}S_{1}(m,n)\frac{z^{m}}{m!}
\end{equation*}%
here we use $S_{1}(m,n)=0$ if $n>m$. After some elementary calculations,
equating the coefficients $z^{m}$ on both sides of the previous equation,
one has the following well known novel formula, which has been proven by
other different methods:%
\begin{equation}
\sum_{n=0}^{m}B_{n}S_{1}(m,n)=\frac{(-1)^{m}m!}{m+1}  \label{a91}
\end{equation}%
(\textit{cf}. \cite{KimDahee}, \cite[p. 45, Exercise 19 (b)]{5Riardon}).

There are many other proof of (\ref{a91}). One of them was given Kim \cite%
{KimDahee} with the $p$-adic invariant integral on the set of $p$-adic
integers. On the other hand, the relation is given by equation the (\ref{a91}%
), Riordan \cite[p. 45, Exercise 19 (b)]{5Riardon} represented it by the
notation $(b)_{n}$.

By using (\ref{D}) and (\ref{a91}), in the next section, we give relations
among the numbers $(b)_{n}$, the Daehee numbers $D_{n}$, and the numbers $%
y(n,\lambda )$.

\subsection{Further remarks and observations on the numbers $y(n,\protect%
\lambda )$}

Substituting $\lambda =2$ into (\ref{ynldef}), in \cite{SimsekMTJPAM2020} we
defined the following combinatorial numbers:%
\begin{equation}
y(n):=y(n,2)=\sum_{j=0}^{n}\frac{(-1)^{n}}{(j+1)2^{j+1}}.  \label{88.8C}
\end{equation}%
In \cite{SimsekMTJPAM2020}, we also give a relationship between the numbers $%
y(n)$ and the $\lambda $-Apostol-Daehee numbers.

Putting $\lambda =1$ in (\ref{ynldef}), we have%
\begin{equation}
y(n,-1)=\frac{1}{2\left( n+1\right) }\sum_{j=0}^{n}\frac{1}{\binom{n}{j}}
\label{SumRecB}
\end{equation}%
(\textit{cf}. \cite{SimsekMTJPAM2020}).

\begin{remark}
Recently many authors have studied on binomial coefficients involving the
following well-known sums and their applications:%
\begin{eqnarray*}
&&\sum_{j=0}^{n}\frac{1}{\binom{n}{j}}, \\
&&\sum_{j=0}^{n}\frac{(-1)^{j}}{\binom{n}{j}}, \\
&&\sum_{j=0}^{n}\binom{n}{j}j^{k}
\end{eqnarray*}%
and so on (cf. \cite{comtet}, \cite{5Riardon}, \cite{Gould7}, \cite%
{KucukogluAnkara}-\cite{KucukogluAADM2019}, \cite{KucukogluAADM2019}\cite%
{Rota}, \cite{SimsekCom.MMAS}-\cite{SimsekMTJPAM2020}, \cite{SuryCOM}).
\end{remark}

\begin{remark}
Most recently, in \cite[Equations. (3), (13), and (14) ]{simsekJMAA}, we
constructed the following generating function for the numbers $%
y_{6}(m,n;\lambda ,p)$ involving finite sums involving higher powers of
binomial coefficients%
\begin{equation*}
_{p}F_{p-1}\left[ 
\begin{array}{c}
-n,-n,...,-n \\ 
1,1,...,1%
\end{array}%
;\left( -1\right) ^{p}\lambda e^{t}\right] =n!\sum_{m=0}^{\infty
}y_{6}(m,n;\lambda ,p)\frac{t^{m}}{m!},
\end{equation*}%
where%
\begin{eqnarray*}
n!y_{6}(0,n;\lambda ,p) &=&_{p}F_{p-1}\left[ 
\begin{array}{c}
-n,-n,...,-n \\ 
1,1,...,1%
\end{array}%
;\left( -1\right) ^{p}\lambda \right] \\
&=&\sum\limits_{k=0}^{n}\binom{n}{k}^{p}k^{m}\lambda ^{k},
\end{eqnarray*}%
and $_{p}F_{p}$ \ denotes the well known generalized hypergeometric function
which is defined by 
\begin{equation*}
_{p}F_{q}\left[ 
\begin{array}{c}
\alpha _{1},...,\alpha _{p} \\ 
\beta _{1},...,\beta _{q}%
\end{array}%
;z\right] =\sum\limits_{m=0}^{\infty }\left( \frac{\prod\limits_{j=1}^{p}%
\left( \alpha _{j}\right) ^{\overline{m}}}{\prod\limits_{j=1}^{q}\left(
\beta _{j}\right) ^{\overline{m}}}\right) \frac{z^{m}}{m!}.
\end{equation*}%
where the above series converges for all $z$ if $p<q+1$, and for $\left\vert
z\right\vert <1$ if $p=q+1$. Assuming that all parameters have general
values, real or complex, except for the $\beta _{j}$, $j=1,2,...,q$ none of
which is equal to zero or a negative integer and also 
\begin{equation*}
\left( \alpha \right) ^{\overline{v}}=\prod\limits_{j=0}^{v-1}(\lambda +j),
\end{equation*}%
and \ $\left( \lambda \right) ^{\overline{0}}=1$ for $\lambda \neq 1$, where 
$v\in \mathbb{N}$, $\lambda \in \mathbb{C}$. For the generalized
hypergeometric function and their applications see for details (\textit{cf}. 
\cite{Koepf}, \cite{simsekJMAA}, \cite{RT}; and references therein).

Here, we noting that (in future studies) the relationships between the
numbers $y(n,\lambda )$ and the numbers $y_{6}(0,n;\lambda ,p)$, which have
potential use in not only in mathematics but also in other areas, may also
be investigated.
\end{remark}

\section{Proofs of main Theorems}

In this chapter, the proofs of the theorems from Theorem \ref{Th.berD} to
Theorem \ref{Th. Fib}, as well as the new novel formulas and relations
derived in the light of the results of these theorems are given.

\begin{proof}[Proof of Theorem \protect\ref{Th.berD}]
From (\ref{ynldef}), we get%
\begin{equation}
y(n,\lambda )=\frac{1}{\lambda }\sum_{j=0}^{n}(-1)^{j+1}\frac{\binom{n+1}{j}%
}{(j+1)!}\frac{1}{B_{j}^{n+1}\left( \lambda \right) }.  \label{a9}
\end{equation}%
Combining (\ref{a9}) with (\ref{a91}) and (\ref{D}), after some elementary
calculations, we arrive at the desired result.
\end{proof}

\begin{proof}[Proof of Theorem \protect\ref{Th. 1a}]
By using (\ref{ynldef}), we give integral representation for the numbers $%
y(n,\lambda )$. We set%
\begin{equation}
y(n,\lambda )=\frac{(-1)^{n}}{\left( \lambda -1\right) ^{n+2}}%
\sum_{j=0}^{n}\int\limits_{0}^{\frac{\lambda -1}{\lambda }}x^{j}dx.
\label{1a}
\end{equation}%
By using the above equation, we also get%
\begin{equation}
y(n,\lambda )=\frac{(-1)^{n}}{\left( \lambda -1\right) ^{n+2}}%
\int\limits_{0}^{\frac{\lambda -1}{\lambda }}\frac{1-x^{n+1}}{1-x}dx.
\label{1b}
\end{equation}%
Combining (\ref{1a}) and (\ref{1b}), proof is completed.
\end{proof}

\begin{proof}[proof of Theorem \protect\ref{Th Hn}.]
Putting $\lambda =\frac{1}{2}$ in (\ref{ynldef}), we have%
\begin{equation}
y\left( n,\frac{1}{2}\right) =2^{n+2}\sum_{j=0}^{n}\frac{(-1)^{j+1}}{j+1}.
\label{91b}
\end{equation}%
We know from Sofo's work \cite[Eq. (1.5)]{1Sofo} that the right-hand side of
the previous equation gives us a formula of the following well known
alternating harmonic numbers, which are related to the Riemann zeta
function, the Polylogarithm (or de Jonquiere's) function, the Psi (or
Digamma) function and the other special functions:%
\begin{equation}
\sum_{j=1}^{n}\frac{(-1)^{j}}{j}=H_{\left[ \frac{n}{2}\right] }-H_{n}.
\label{91d}
\end{equation}%
Combining (\ref{91b}) with (\ref{91d}), proof of Theorem \ref{Th Hn} is
completed.
\end{proof}

\begin{proof}[Proof of Theorem \protect\ref{Th.Int}]
Using (\ref{91b}), we get%
\begin{equation}
y\left( n,\frac{1}{2}\right)
=-2^{n+2}\sum_{j=0}^{n}(-1)^{j}\int\limits_{0}^{1}x^{j}dx.  \label{91f}
\end{equation}%
Combining (\ref{91f}) with the following well known formula 
\begin{equation*}
\sum_{j=0}^{n}(-x)^{j}=\frac{1+(-1)^{n}x^{n+1}}{1+x},
\end{equation*}%
we obtain%
\begin{equation*}
y\left( n,\frac{1}{2}\right) =\frac{1}{2^{n+2}}\int\limits_{0}^{1}\frac{%
1+(-1)^{n}x^{n+1}}{1+x}dx.
\end{equation*}

This last equation shows us that the proof of the theorem completed.
\end{proof}

\begin{proof}[Proof of Theorem \protect\ref{Th.Int2}]
Combining (\ref{91c}) with (\ref{91b1}), we get%
\begin{equation*}
y\left( n,\frac{1}{2}\right) =2^{n+2}\left( \int\limits_{0}^{1}\frac{1-x^{%
\left[ \frac{n}{2}\right] }}{1-x}dx-\int\limits_{0}^{1}\frac{1-x^{n}}{1-x}dx+%
\frac{(-1)^{n+1}}{n+1}\right) .
\end{equation*}%
After some elementary calculations in the above equation, we arrive at the
desired result.
\end{proof}

\begin{proof}[Proof of Theorem \protect\ref{Th. Fib}]
Putting $\lambda =\frac{1+\sqrt{5}}{2}$ into (\ref{ynldef}), we get%
\begin{equation*}
y\left( n,\frac{1+\sqrt{5}}{2}\right) =\sum_{j=0}^{n}\frac{(-1)^{n}}{%
(j+1)\left( \frac{1+\sqrt{5}}{2}\right) ^{j+1}\left( \frac{1+\sqrt{5}}{2}%
-1\right) ^{n+1-j}}.
\end{equation*}%
Combining the previous equation with the following well known identities
involving the Fibonacci numbers:%
\begin{equation*}
\left( \frac{1+\sqrt{5}}{2}\right) ^{j+1}=\frac{1+\sqrt{5}}{2}F_{j+1}+F_{j}
\end{equation*}%
(\textit{cf}. \cite{4Koshy}), after some elementary calculations, we arrive
at the desired result.
\end{proof}

\section{Formulas for the numbers $y\left( n,\frac{1}{2}\right) $ derived
from main Theorems}

In this section we give some formulas which are derived from main theorems.

In \cite{3Euler}, \cite{2Euler}, and \cite{1Sofo}, the Harmonic numbers are
given by the following integral representation:%
\begin{equation}
H_{n}=\int\limits_{0}^{1}\frac{1-x^{n}}{1-x}dx,  \label{91b1}
\end{equation}%
and%
\begin{equation}
H_{n}=-n\int\limits_{0}^{1}x^{n-1}\ln
(1-x)dx=-n\int\limits_{0}^{1}(1-x)^{n-1}\ln (x)dx,  \label{91b2}
\end{equation}%
where $H_{0}=0$.

Combining (\ref{91c}) with (\ref{91b2}), we arrive at the following theorem:

\begin{theorem}
Let $n\in \mathbb{N}_{0}$. Then we have%
\begin{eqnarray*}
y\left( n,\frac{1}{2}\right) &=&2^{n+2}\left( n\int\limits_{0}^{1}\left(
1-x\right) ^{n-1}\ln (x)dx-\left[ \frac{n}{2}\right] \int\limits_{0}^{1}%
\left( 1-x\right) ^{\left[ \frac{n}{2}\right] -1}\ln (x)dx\right) \\
&&+\frac{(-1)^{n+1}2^{n+2}}{n+1}.
\end{eqnarray*}
\end{theorem}

We now give relation among the numbers $y\left( n,\frac{1}{2}\right) $, the
Psi (or Digamma) function and the Euler constant.

The Euler's constant (or Euler-Mascheroni) constant is given by%
\begin{equation*}
\gamma =\lim_{m\rightarrow \infty }\left( -\ln (m)+\sum_{j=1}^{m}\frac{1}{j}%
\right)
\end{equation*}%
and the Psi (or Digamma) function%
\begin{equation*}
\psi (z)=\frac{d}{dz}\left\{ \log \Gamma (z)\right\} ,
\end{equation*}%
where $\Gamma (z)$ denotes the Euler gamma function, which is defined by%
\begin{equation*}
\Gamma (z)=\int\limits_{0}^{\infty }t^{z-1}e^{-t}dt
\end{equation*}%
where $z=x+iy$ with $x>0$. For $z=n\in \mathbb{N}$, $\Gamma (n+1)=n!$ (%
\textit{cf}. \cite{3Euler}, \cite{2Euler}, \cite{SrivatavaChoi}).

Sofo \cite{1Sofo} gave the following formula:%
\begin{equation}
H_{n}=\gamma +\psi (n+1)  \label{91b3}
\end{equation}%
where $H_{0}=0$.

Combining (\ref{91c}) with (\ref{91b3}), we arrive at the following theorem:

\begin{theorem}
Let $n\in \mathbb{N}_{0}$. Then we have%
\begin{equation*}
y\left( n,\frac{1}{2}\right) =2^{n+2}\left( \frac{(-1)^{n+1}}{n+1}+H_{\left[ 
\frac{n}{2}\right] }-\gamma -\psi (n+1)\right) .
\end{equation*}
\end{theorem}

Combining (\ref{a9}) with (\ref{a91}), we also arrive at the following
result.

\begin{corollary}
Let $n\in \mathbb{N}_{0}$. Then we have%
\begin{equation}
y(n,\lambda )=\frac{-1}{\lambda }\sum_{j=0}^{n}\binom{n+1}{j}\frac{(b)_{j}}{%
j!B_{j}^{n+1}\left( \lambda \right) }  \label{91bB}
\end{equation}%
or%
\begin{equation}
y(n,\lambda )=\frac{-1}{\lambda }\sum_{j=0}^{n}\binom{n+1}{j}\frac{D_{j}}{%
j!B_{j}^{n+1}\left( \lambda \right) }.  \label{91bBb}
\end{equation}
\end{corollary}

\section{Conclusion}

In this study, some properties of the numbers $y\left( n,\lambda \right) $
with special finite sums containing these numbers have been found. Many new
novel formulas and relations blended with well known special functions
including the Euler gamma function, the Euler-Mascheroni constant, the Psi
function, special polynomials including Bernstein basis functions, and
special numbers including the Fibonacci number, the Bernoulli numbers, the
Euler numbers, the Stirling numbers, and the Daehee numbers, those of known
by the names of well-known mathematicians.

On the other hand, new formulas and integral representations of the numbers $%
y\left( n,\lambda \right) $ have been proved by using integral
representations of the harmonic numbers and the alternating harmonic
numbers, which are based on very old studies and have rich properties and
relationships, and formulas related to other special functions, involving
the Digamma function, the Euler gamma function, and Euler constant.

In \cite{SimsekMTJPAM2020}, we gave some open problems for the numbers $%
y(n,\lambda )$. In section 1, we've made some efforts to find solutions some
of them throughout this study.

In your next studies, it is planned to investigate the solutions of the
above open questions, including the numbers $y(n,\lambda )$.

\end{document}